\documentclass[12pt]{article}

\usepackage{tikz}
\usepackage{jfe_preambule}
\usepackage{booktabs}
\usepackage{bbm}

\newcommand{\norm}[1]{\left\lVert#1\right\rVert}

\begin{document}

\setlist{noitemsep}  

\title{T-statistic for Autoregressive process}

\author{Eric Benhamou \thanks{A.I. SQUARE CONNECT, 35 Boulevard d'Inkermann 92200 Neuilly sur Seine, France}  
\textsuperscript{,} 
\thanks{LAMSADE, Université Paris Dauphine, Place du Maréchal de Lattre de Tassigny,75016 Paris, France} 
\textsuperscript{,} 
\thanks{E-mail: eric.benhamou@aisquareconnect.com, eric.benhamou@dauphine.eu}
}

\date{}              

\singlespacing

\maketitle

\vspace{-.2in}
\begin{abstract}
\noindent In this paper, we discuss the distribution of the t-statistic under the assumption of normal autoregressive distribution for the underlying discrete time process.
This result generalizes the classical result of the traditional t-distribution where the underlying discrete time process follows an uncorrelated normal distribution. However, for AR(1), the underlying process is correlated. All traditional results break down and the resulting t-statistic is a new distribution that converges asymptotically to a normal. We give an explicit formula for this new distribution obtained as the ratio of two dependent distribution (a normal and the distribution of the norm of another independent normal distribution). 
We also provide a modified statistic that is follows a non central t-distribution. Its derivation comes from finding an orthogonal basis for the the initial circulant Toeplitz covariance matrix. Our findings are consistent with the asymptotic distribution for the t-statistic derived for the asympotic case of large number of observations or zero correlation.
This exact finding of this distribution has applications in multiple fields and in particular provides a way to derive the exact distribution of the Sharpe ratio under normal AR(1) assumptions.
\end{abstract}

\medskip



\thispagestyle{empty}

\clearpage

\onehalfspacing
\setcounter{footnote}{0}
\renewcommand{\thefootnote}{\arabic{footnote}}
\setcounter{page}{1}


\section{Introduction}
Let $X_1, \ldots, X_n$ be a random sample from a cumulative distribution function (cdf) $F(·)$ with a constant mean $\mu$ and let define the following statistics referred to as the t-statistic
\begin{equation}\label{tstatistic} 
T_n = T(X_n) = \frac{\sqrt{n} ( \bar X_n - \mu ) }{s_n}
\end{equation}
where $\bar X_n $ is the empirical mean, $s_n^2$ the empirical Bessel corrected empirical variance, and $X_n$ the regular full history of the random sample defined by: 
\begin{equation}
\bar{X}_n =\frac{1}{n}\sum_{i=1}^{n}X_i, \quad s_n^2 = \frac{1}{n-1}\sum_{i=1}^{n}(X_i - \bar{X}_n)^2, \quad X_n = (X_1, \ldots, X_n)^T
\end{equation}

It is well known that if the sample comes from a normal distribution, $N(0, \sigma)$, $T_n$ has the Student t-distribution with $n - 1$ degrees of freedom. The proof is quite simple (we provide a few in the appendix section in  \ref{t_student}). If the variables $(X_i)_{i=1..n}$ have a mean $\mu$ non equal to zero, the distribution is referred to as a non-central t-distribution with non centrality parameter given by 
\begin{equation}
\eta = \sqrt n \quad \frac{\mu}{\sigma}
\end{equation}
Extension to weaker condition for the t-statistics has been widely studied. 

Mauldon \cite{Mauldon_1956} raised the question for which pdfs the t-statistic as defined by \ref{tstatistic} is t-distributed with $d - 1$ degrees of freedom. Indeed, this characterization problem can be generalized to the one of finding all the pdfs for which a certain statistic possesses the property which is a characteristic for these pdfs. \cite{Kagan_1973}, \cite{Bondesson_1974} and \cite{Bondesson_1983} to cite a few tackled Mauldon’s
problem. \cite{Bondesson_1983} proved the necessary and sufficient condition for a t-statistic to have Student’s t-distribution with $d - 1$ degrees
of freedom for all sample sizes is the normality of the underlying distribution. It is not necessary that $X_1,...,X_n$  is an independent sample. Indeed consider
$X_1,...,X_n$  as a random vector $X_n = (X_1,...,X_n)^T$ each component of which having the
same marginal distribution function, $F(·)$. \cite{Efron_1969} has pointed out that the weaker condition
of symmetry can replace the normality assumption. Later, \cite{Fang_2001}  showed that if the vector $X_n$ has a spherical distribution, then the
t-statistic has a t-distribution. A natural question that gives birth to this paper was to check if the Student resulting distribution is conserved in the case of an underlying process $(X_i)_{i=1,\ldots }$ following an AR(1) process. This question and its answer has more implication than a simple theoretical problem. Indeed, if by any chance, one may want to test the statistical significance of a coefficient in a regression, one may do a t-test and rely upon the fact that the resulting distribution is a Student one. If by any chance, the observations are not independent but suffer from auto-correlation, the building blocks supporting the test break down. Surprisingly, as this problem is not easy, there has been few research on this problem. Even if this is related to the Dickey Fuller statistic (whose distribution is not closed form and needs to be computed by Monte Carlo simulation), this is not the same statistics. \cite{Mikusheva_2015} applied an Edgeworth expansion precisely to the Dickey Fuller statistics but not to the original t-statistic. The neighboring Dickey Fuller statistic has the great advantage to be equal to the ratio of two known continuous time stochastic process making the problem easier. In the sequel, we will first review the problem, comment on the particular case of zero correlation and the resulting consequence of the t-statistic. We will emphasize the difference and challenge when suddenly, the underlying observations are not any more independent. We will study the numerator and denominator of the t-statistic and derive their underlying distribution. We will in particular prove that it is only in the case of normal noise in the underlying AR(1) process, that the numerator and denominator are independent. We will then provide a few approximation for this statistic and conclude.

\section{AR(1) process}
The assumptions that the underlying process (or observations) $(X_i)_{i=1,\ldots; n }$ follows an AR(1) writes : 
\begin{equation}\label{AR_assumptions}
\left\{ {
\begin{array}{l l l l }
X_t 			&  = & \mu + \epsilon_t 				& \quad t \geq 1 ; \\
\epsilon_t 	& = & \rho \epsilon_{t-1} + \sigma v_t 	& \quad t \geq 2 ; 
\end{array} } \right.
\end{equation}
where $v_t$ is an independent white noise processes (i.i.d. variables with zero mean and unit constant variance). To assume a stationary process, we impose
\begin{equation}
\lvert {\rho} \rvert  \leq 1
\end{equation}

It is easy to check that equation \ref{AR_assumptions} is equivalent to
\begin{equation}
\begin{array}{l l l l }
X_t 			&  = & \mu + \rho ( X_{t-1} - \mu )+ \sigma v_t 	& \quad t \geq 2 ; 
\end{array}
\end{equation}

We can also easily check that the variance and covariance of the returns are given by
\begin{equation}\label{moment2}
\begin{array}{l l l l }
V(X_t) & = & \frac {\sigma^2} {1-\rho^2} 								\; \;\;\;\; \text{for} \;  t \geq 1  \\
Cov(X_t, X_u ) & = & \frac {\sigma^2 \rho^{ \lvert {t -u} \rvert  }} {1-\rho^2} 	\; \;\;\;\; \text{for} \;  t,u \geq 1 
\end{array}
\end{equation}

Both expressions in \ref{moment2} are independent of time $t$ and the covariance only depends on $\lvert {t -u} \rvert$ implying that $X_t$ is a stationary process.

\subsection{Case of Normal errors}
If in addition, we assume that $v_t$ are distributed according to a normal distribution, we can fully characterize the distribution of $X$ and rewrite our model in reduced matrix formulations as follows:
\begin{equation}\label{reduced_matrix}
X = \left(  \begin{array}{c} X_1 \\ \vdots \\ X_n \end{array} \right)  = \mu \cdot \mathbbm{1}_n  + \sigma \cdot  \epsilon = \mu \left(  \begin{array}{c} 1 \\ \vdots \\ 1 \end{array} \right) + \sigma \left(  \begin{array}{c} \epsilon_1 \\ \vdots \\ \epsilon_n  \end{array} \right) 
\end{equation}
where $\epsilon \sim N \left( 0, \Omega = \left( \frac{ \rho ^{ | i - 1 | }}{1-\rho^2}\right)_{ij} \right)$, 
hence, $X \sim N \left( \mu \cdot  \mathbbm{1}_n,  \sigma^2 \Omega \right)$. 

The $\Omega$ matrix is a Toeplitz circulant matrix defined as

\begin{equation}\label{Omega}
\Omega = \frac{1}{1 - \rho^2}   \left( 
\begin{array}{l l l l l}
{1}					& {\rho}		& \ldots 		&  {\rho^{n-2}} 	& {\rho^{n-1}}	\\
{\rho}  				&  1 		& \ldots 		& {\rho^{n-3}} 	&  {\rho^{n-2}}	\\
\vdots 				& \vdots 		& \ddots 		& \vdots 		  	& \vdots 			\\
{\rho^{n-2}} 	&  {\rho^{n-3}} 		& \ldots 		& 1				& {\rho}			\\
{\rho^{n-1}} 	&  {\rho^{n-2}} 		& \ldots 		& {\rho}			&{1}	
\end{array}
\right) = M^T M
\end{equation}

Its Chlolesky decomposition is given by
\begin{eqnarray}\label{OmegaSqrt}
M &= &\frac{1}{\sqrt{ 1- \rho^2} }  \left( 
\begin{array}{l l l l l}
{1}			& 0							& \ldots 		&  0					& 0 \\
{\rho}		& \sqrt{1 -\rho^2}				& \ldots 		& 0			 		& 0	\\
\vdots 		& \vdots 						& \ddots 		& \vdots 				& \vdots 	\\
{\rho^{n-2}} 	&  \rho^{n-3} \sqrt{1 -\rho^2} 	& \ldots 		& \sqrt{1 -\rho^2} 	& 0 \\
 {\rho^{n-1}}	&  \rho^{n-2} \sqrt{1 -\rho^2} 	& \ldots 		& \rho \sqrt{1 -\rho^2}	& \sqrt{1 -\rho^2}	
\end{array}
\right)  
\end{eqnarray}
It is worth splitting $M$ into $I_n$ and another matrix as follows:
\begin{eqnarray}
M &= &\left( 
\begin{array}{l l l l l}
1+\frac{1-\sqrt{1-\rho^2}}{\sqrt{1-\rho^2}}	& 0				& \ldots 		&  0			& 0 \\
\frac{\rho}{\sqrt{1-\rho^2}}					& 1				& \ldots 		& 0			& 0	\\
\vdots 									& \vdots 			& \ddots 		& \vdots 		& \vdots 	\\
\frac{\rho^{n-2}}{\sqrt{1-\rho^2}}			&  \rho^{n-3} 	& \ldots 		& 1			& 0 \\
\frac{\rho^{n-1}}{\sqrt{1-\rho^2}}			&  \rho^{n-2} 	& \ldots 		& \rho 		& 1
\end{array}
\right) = \underbrace{I_n + \left( 
\begin{array}{l l l l l}
\frac{1-\sqrt{1-\rho^2}}{\sqrt{1-\rho^2}}		& 0				& \ldots 		& 0			& 0 \\
\frac{\rho}{\sqrt{1-\rho^2}}					& 0				& \ldots 		& 0			& 0	\\
\vdots 									& \vdots 			& \ddots 		& \vdots 		& \vdots 	\\
\frac{\rho^{n-2}}{\sqrt{1-\rho^2}}			&  \rho^{n-3} 	& \ldots 		& 0			& 0 \\
\frac{\rho^{n-1}}{\sqrt{1-\rho^2}}			&  \rho^{n-2} 	& \ldots 		& \rho 		&  0
\end{array}
\right) }_{I_n \; + \; \hspace{2.3cm}  N \hspace{3cm}}
\end{eqnarray}

The inverse of $\Omega$ is given by
\begin{equation}\label{OmegaInv}
A = \Omega^{-1} = \left( 
\begin{array}{l l l l l}
{1}			& {-\rho}			& \ldots 		& 0				& 0	\\
{-\rho}		& {1+ \rho ^2 } 	& \ldots 		& 0			 	& 0	\\
\vdots 		& \vdots 			& \ddots 		& \vdots 			& \vdots 	\\
0			&  0			 	& \ldots 		& {1+ \rho ^2 } 	& {-\rho}	\\
0			&  0			 	& \ldots 		&   {-\rho}		& 1 	
\end{array}
\right) = L^{T} L
\end{equation}

Its Cholesky decomposition $L$ is given by
\begin{equation}\label{OmegaInvSqrt}
L = \left( 
\begin{array}{l l l l l}
\sqrt{1 -\rho^2}	& 0				& \ldots 		& 0				& 0	\\
{-\rho}			& 1			 	& \ldots 		& 0			 	& 0	\\
\vdots 			& \vdots 			& \ddots 		& \vdots 			& \vdots 	\\
0				&  0			 	& \ldots 		& 1				& 0 \\
0				&  0			 	& \ldots 		& {-\rho}			& 1
\end{array}
\right)
\end{equation}

Notice in the various matrix the dissymmetry between the first term and the rest. This shows up  for instance in the first diagonal term of $L$ which is $\sqrt{1 -\rho^2}$, while all other diagonal terms are equal to 1. Similarly, in the matrix $N$, we can notice that the first column is quite different from the other ones as it is a fraction over $\sqrt{1 -\rho^2}$.

\subsection{T-statistics issue}
The T-statistic given by equation \ref{tstatistic} is not easy to compute. For the numerator, we have that $ \bar X_n - \mu$ follows a normal distribution. The proof is immediate as $\bar X_n $ is a linear combination of the Gaussian vector generated by the AR(1) process. We have $\bar X_n = \frac{1}{n} \mathbbm{1}_n \cdot X$. It follows that $\bar X_n  \sim N( \mu,  \frac{ \sigma^2 }{n^2} \mathbbm{1}_n^T \cdot \Omega \cdot \mathbbm{1}_n )$ (for a quick proof of the fact that any linear combination of a Gaussian vector is normal, see \ref{corr_normal}). In section \ref{Computations}, we will come back on the exact computation of the characteristics of the distribution of the numerator and denominator as this will be useful in the rest of the paper.

As for the denominator, for a non null correlation $\rho$, the distribution of $s_n^2$ is not a known distribution. 
\\\\
The distributions of the variables $\left(Y_i=X_i - \bar{X}_n \right)_{i=2, \ldots, n}$ are normal given by
 $ Y_ i \sim N( 0, \sigma_{Y_i} )$ with 
$\sigma_{Y_i}  = (\delta_i - \frac 1 n  \mathbbm{1}_n)^T \cdot  \Omega \cdot  (\delta_i - \frac 1 n  \mathbbm{1}_n)$. 
where $ \delta_i  = \underbrace{\left( 0, 0, \ldots, 1, \ldots, 0, 0 \right)^T }_{\text{1 at ith position}}$

Hence the square of these normal variables $Z_i=Y_i ^2$ is the sum of Gamma distributions. However, we cannot obtain a closed form for the distribution $s_n^2$ as the variance of the different terms are different and the terms are neither independent. If the correlation is null, and only in this specific case, we can apply the Cochran’s Theorem to prove that $s_n^2$ follows a Chi square distribution with $n-1$ degree of freedom. However, in the general case, we need to rely on approximation that will be presented in the rest of the paper.

Another interesting result is to use the Cholesky decomposition of the inverse of the covariance matrix of our process to infer a modified t-statistic that has now independent terms and is defined as follows

Let us take the modified process defined by
\begin{equation}
U = L  X
\end   {equation}
The variables $U$ is distributed according to a normal $U  \sim N( \mu L  \mathbbm{1},  Id_n )$. We can compute the modified T-statistic $\tilde{T}_n$ on $U$ as follows:

\begin{equation}\label{tstatistic2} 
\tilde{T}_n = \frac{\sqrt{n} ( \bar U_n - \mu ) }{\tilde{s}_n}
\end{equation}
where 
\begin{equation}
\bar{U}_n =\frac{1}{n}\sum_{i=1}^{n}U_i, \quad \tilde{s}_n^2 = \frac{1}{n-1}\sum_{i=1}^{n}(U_i - \bar{U}_n)^2
\end{equation}

In this specific case, the distribution of $\tilde{T}_n$ is a Student distribution of degree $n-1$. We will now work on the numerator and denominator of the T-statistic in the specific case of AR(1) with a non null correlation $\rho$.

\section{Expectation and variance of numerator and denominator}\label{Computations}
The numerator of the T-statistic writes
\begin{eqnarray}
\sqrt{n} (\bar{X}_n - \mu) =\frac{1}{\sqrt{n}}\sum_{i=1}^{n}( X_i - \mu),
\end{eqnarray}

Its expectation is null as each term is of zero expectation. Its variance is given by

\begin{lemma}\label{lemma_var_num}
\begin{eqnarray}
\text{Var}(\sqrt n (\bar{X}_n-\mu)) &=&  \frac{ \sigma^2}{1 - \rho^2} \left[ \frac{ 1 + \rho}{1-\rho} - \frac{2 \rho (1 -  \rho^{n}) }{n(1-\rho)^2}  \right]  \label{lemma_var_num_eq1} \\
& =&  \frac{ \sigma^2}{ (1 - \rho)^2} \left[ 1 - \frac{2 \rho (1 -  \rho^{n}) }{n(1-\rho)(1+\rho)}  \right]   \label{lemma_var_num_eq2}
\end{eqnarray}
\end{lemma}

\begin{proof}: See \ref{proof_var_num}
\end{proof}

The proposition \ref{lemma_var_num} is interesting as it states that the sample mean variance converges to $\cfrac{ \sigma^2}{(1 - \rho)^2}$  \quad for large $n$. It is useful to keep the two forms of the variance. The first one (equation (\ref{lemma_var_num_eq1})) is useful in following computation as it shares the denominator term $1 - \rho^2$. The second form (equation \ref{lemma_var_num_eq2}) gives the asymptotic form.

The denominator writes:
\begin{eqnarray}
s_n = \sqrt{ \frac{1}{n-1}\sum_{i=1}^{n}(X_i - \bar{X}_n)^2},
\end{eqnarray}

In the following, we denote by $Y_i = X_i - \mu$ the zero mean variable and work with these variables to make computation easier. We also write $Y_i^{\perp}$ the variable orthogonal to $Y_i$ whose variance (we sometimes refer to it as its squared norm to make notation easier) is equal to the one of $Y_i$ : $\norm{Y_i}^2 =  \norm{Y_i^{\perp}}^2$. To see the impact of correlation, we can write for any $j>i$, $Y_j=\rho^{j-i} Y_i + \sqrt{ 1 - \rho^{2(j-i)} } Y_i^{\perp}$. 

As studying this denominator is not easy because of the presence of the square root, it is easier to investigate the properties of its squared given by

\begin{eqnarray}
s_n^2 = \frac{\sum_{i=1}^{n}(Y_i - \bar{Y}_n)^2}{n-1}= \frac{\sum_{i=1}^{n}Y_i^2 - n\bar{Y}_n^2}{n-1}
\end{eqnarray}

We have that the mean of $\bar{Y}_n$ is zero while proposition \ref{lemma_var_num} gives its variance :
\begin{eqnarray}
\text{Var}(\bar{Y}_n) & =&   \frac{ \sigma^2}{n (1 - \rho^2)} \left[ \frac{ 1 + \rho}{1-\rho} - \frac{2 \rho (1 -  \rho^{n}) }{n(1-\rho)^2}  \right]   =   \frac{ \sigma^2}{ n (1 - \rho)^2} \left[ 1 - \frac{2 \rho (1 -  \rho^{n}) }{n(1-\rho)(1+\rho)}  \right] 
\end{eqnarray}

\begin{lemma}\label{lemma_covar}
The covariance between $\bar{Y}_n$ and each stochastic variable $Y_j$ is useful and given by
\begin{eqnarray}\label{lemma_covar_eq1}
\text{Cov}(\bar{Y}_n, Y_j) =  \frac{ \sigma^2}{ n (1 - \rho^2)} \left[ \frac{ 1 + \rho -\rho^{n+1-j}-\rho^{j}} {1-\rho}   \right]  
\end{eqnarray}
In addition, we have a few remarkable identities
\begin{eqnarray}\label{lemma_covar_eq2}
\sum_{j=1}^n \text{Cov}(\bar{Y}_n, Y_j) =  \frac{ \sigma^2}{ (1 - \rho^2)} \left[ \frac{ 1 + \rho}{1-\rho} - \frac{2 \rho (1 -  \rho^{n}) }{n(1-\rho)^2}  \right]
\end{eqnarray}

\begin{eqnarray}\label{lemma_covar_eq3}
\sum_{j=1}^n \left( \text{Cov}(\bar{Y}_n, Y_j) \right)^2=  \frac{ \sigma^4}{  (1 - \rho^2)^2}  \left[ \frac{ (1+\rho)^2 + 2 \rho^{n+1} } { (1-\rho)^2 } \frac{1}{n} 
-  \frac{4 (1+\rho)^2 \rho (1-\rho^n) - 2 \rho^2 (1-\rho^{2n})}  {(1-\rho)^2 (1-\rho^2)} \frac{1}{n^2} \right]
\end{eqnarray}
\end{lemma}

\begin{proof}: See \ref{proof_covar}
\end{proof}

We can now compute easily the expectation and variance of the denominators as follows

\begin{proposition}\label{prop_denom_expectation}
The expectation of $s_n^2$ is given by:
\begin{eqnarray}
\mathbb{E} {s_n^2} = \frac{ \sigma^2}{1 - \rho^2}   \left( 1 - \frac{ 2 \rho }{(1-\rho) (n-1)} + \frac{ 2 \rho( 1-\rho^{n}) }{n (n-1) (1-\rho)^2}  \right)  
\end{eqnarray}
\end{proposition}

\begin{proof}: See \ref{proof_denom_expectation}
\end{proof}

\begin{proposition}\label{second_moment_denom}
The second moment of $s_n^2$ is given by:
\begin{align}
\mathbb{E}[s_n^4] & = \frac{\sigma^4}{(1-\rho^2)^2} \frac{1}{(n-1)^2} \left[ n^2-1 + \rho \left(n A_1 + A_2 + \frac{1}{n} A_3 + \frac{1}{n^2} A_4 \right) \right]
\end{align}
with 
\begin{align}
A_1 & = \frac{-4}{1-\rho^2} \\
A_2 & = \frac{- 2 \left(3 + 9 \rho + 11 \rho^2 + 3 \rho^3 + 6 \rho^n + 12 \rho^{n+1} + 6 \rho^{n+2}-2 \rho^{2n+2}\right)} {(1-\rho^2)^2} \\
A_3 & = \frac{ 4 (1 - \rho^n) (1 - 3 \rho + 4 \rho^2 - 8 \rho^{n+1})}{(1 -r)^3 (1 + r)} \\
A_4 & = \frac{12  \rho (1-\rho^{n})^2}{(1-\rho)^4} 
\end{align}
\end{proposition}

\begin{proof}: See \ref{proof_second_moment}
\end{proof}

Combining the two results leads to
\begin{proposition}\label{second_moment_denom}
The variance of $s_n^2$ is given by:
\begin{align}
\text{Var}[s_n^4] & = \frac{\sigma^4}{(1-\rho^2)^2} \frac{1}{(n-1)} \left[ 2 + \frac{ \rho}{n-1}  \left( n B_1 + B_2 + \frac{1}{n} B_3 + \frac{1}{n^2} B_4 \right) \right]
\end{align}
with 
\begin{align}
B_1 & = \frac{-2}{1+ \rho} \\
B_2 & =-\frac{2}{1-\rho } -\frac{4 \rho^2}{(1-\rho )^2} -\frac{2 \left(1-\rho ^n\right)}{(1-\rho )^2} \\
& -\frac{2 \left(12 \rho^{n+1}+6 \rho ^{n+2}-2 \rho ^{2 n+2}+6 \rho ^n+3 \rho ^3+11 \rho^2+9 \rho +3\right)}{\left(1-\rho ^2\right)^2} \\
B_3 & = \frac{ (1-\rho^n)(13- 4 \rho+15 \rho^2-\rho^n - 32 \rho^{n+1} + \rho^{n+2})  }{(1 -r)^3 (1 + r)} \\
B_4 & = \frac{- 4 (1-3\rho) (1-\rho^{n})^2}{(1-\rho)^4} 
\end{align}
\end{proposition}

\begin{proof}: See \ref{proof_var_moment}
\end{proof}

It is worth noting that a direct approach as explained in \cite{Benhamou_2018_SampleVariance} could also give the results for the first, second moments and variance for the numerator and denominator.

\section{Resulting distribution}
The previous section shows that under the AR(1) assumptions, the t-statistic is no longer a Student distribution but the ratio of a normal whose first and second moments have been given above and the norm of a Gaussian whose moments have also been provided. To go further, one need to rely on numerical integration. This is the subject of further research.


\section{Conclusion}
In this paper, we have given the explicit first, second moment and variance of the numerator of the t statistic under the assumption of AR(1) underlying process. We have seen that these moments are very sensitive to the correlation $\rho$ assumptions and that the distribution is far from a Student distribution.

\clearpage

\appendix
\section{Various Proofs for the Student density}
\subsection{Deriving the t-student density} \label{t_student}
Let us first remark that in the T-statistic, the $\sqrt n$ factor cancels out to show the degree of freedom $\sqrt{n-1}$ as follows:
\begin{equation}
T_n = \frac{\bar{X}\,-\,\mu}{s_n/\sqrt{n}} =  \frac{\bar{X}\,-\,\mu}{\frac{\sigma}{\sqrt{n}}} \frac{1}{\frac{s_n}{\sigma}}  = U \,\frac{1}{\frac{s_n}{\sigma}} = \sqrt{n-1} \frac{U}{\sqrt{\frac{\sum(X_i-\bar X)^2}{\sigma^2}}} = \sqrt{n-1} \frac{U}{V}
\end{equation}

In the above expression, it is well know that if $X\sim \,\,\small N(\mu, \sigma)$, then the renormalized variable $U=\frac{(\bar{X}-\mu)}{\sigma/\sqrt{n}}\,\,\sim \,\,\small N(0,1)$ and $V = \sqrt{\frac{\sum(X_i-\bar X)^2}{\sigma^2}}  \sim \,\,\small \chi_{(n-1)}^2$ as well as $U$ and $V$ are independent. Hence, we need to prove that the distribution of $T= U/{\sqrt{V/k}}$ is a Student distribution with $U\sim N(0,1),$ and $V\sim\chi^2_k$ mutually independent, and $k$ is the degree of freedom of the chi squared distribution. \\
\par

The core of the proof relies on two steps that can be proved by various means. \\
\textbf{Step 1} is to prove that the distribution of $T$  is given by
\begin{equation}\label{t_student:step1_eq1}
f_T(t)  = \frac{1}{\Gamma(\frac{k}{2}) 2^{\frac{k+1}{2}}\sqrt{\pi k}} \int_0^\infty e^{-w (\frac{t^2}{2 k}+\frac{1}{2})} w^{\frac{ k-1}{2} }  dw 
\end{equation} 
\textbf{Step 2} is to compute explicitly the integral in equation \ref{t_student:step1_eq1} \\
Step 1 can be done by  transformation theory using the Jacobian of the inverse transformation or the property of the ratio distribution. Step 2 can be done by Gamma function, Gamma distribution properties, Mellin transform or Laplace transform.

\subsection{Proving step 1}
\subsubsection{Using  transformation theory}
The joint density of $U$ and $V$ is:
\begin{equation}
f_{U,V}(u,v) = \underbrace{\frac{1}{(2\pi)^{1/2}} e^{-u^2/2}}_{\text{pdf } N(0,1)}\quad \underbrace{\frac{1}{\Gamma(\frac{k}{2})\,2^{k/2}}\,v^{(k/2)-1}\, e^{-v/2}}_{\text{pdf }\chi^2_k}
\end{equation}
with the distribution support given by $-\infty  < u  < \infty$ and $0 < v  < \infty$. 

Making the transformation $t=\frac{u}{\sqrt{v/ k}}$ and $w=v$, we can compute the inverse: $u=t\,\left(\frac{w}{k}\right)^{1/2}$ and $v=w$.
The Jacobian \footnote{determinant of the Jacobian matrix of the transformation} is given by
\begin{equation}
J(t,w)=  
\begin{vmatrix}
\left(\frac{w}{k}\right)^{1/2} & \frac{t}{ 2 \left( k w \right)^{1/2} }\\
0 & 1
\end{vmatrix}
\end{equation}
\noindent whose value is $(w/ k)^{1/2}$. The marginal pdf is therefore given by:
\begin{eqnarray}
f_T(t) & = & \displaystyle\int_0^\infty \,f_{U,V}\bigg(t\,(\frac{w}{k})^{1/2},w\bigg) J(t,w) \,\mathrm{d} w \\ 
& = & \displaystyle\int_0^\infty 
\frac{1}{(2\pi)^{1/2}} e^{-(t^2 \, \frac{w}{k})  /2}
\frac{1}{\Gamma(\frac{k}{2})\,2^{k/2}} w^{(k/2)-1}\, e^{-w/2}
(w/ k )^{1/2}  \,\mathrm{d} w \\ 
&  = & \frac{1}{\Gamma(\frac{k}{2}) 2^{\frac{k+1}{2}}\sqrt{\pi k}} \int_0^\infty e^{-w (\frac{t^2}{2 k}+\frac{1}{2})} w^{\frac{ k-1}{2} }  dw 
\end{eqnarray}
which proves the result \qed \\\\

\subsubsection{Using ratio distribution}
The square-root of $V$, $\sqrt V \equiv \hat V$ is distributed as a chi-distribution with $k$ degrees of freedom, which has density
\begin{equation}
f_{\hat V}(\hat v) = \frac {2^{1-\frac k 2}}{\Gamma\left(\frac {k}{2}\right)} \hat v^{k-1} \exp\Big \{{-\frac {\hat v^2}{2}} \Big\} \label{t_student:step1_eq21}
\end{equation}

\noindent Define $X \equiv \frac {\hat V}{\sqrt k}$. Then by change-of-variable, we can compute the density of $X$:
\begin{eqnarray}
f_{X}(x) 	&= & f_{\hat V}(\sqrt k x)\Big |\frac {\partial \hat V}{\partial X} \Big|  \\
		&= & \frac {2^{1-\frac k 2}}{\Gamma\left(\frac {k}{2}\right)} k^{\frac k 2}x^{k-1} \exp\Big \{{-\frac {k \, x^2}{2}}  \Big\} \label{t_student:step1_eq22}
\end{eqnarray}

The student's t random variable defined as $T = \frac {Z} {X}$ has a distribution given by the ratio distribution:
\begin{equation}
f_T(t) = \int_{-\infty}^{\infty} |x|f_U(xt)f_X(x)dx   
\end{equation}

\noindent We can notice that $f_X(x) = 0$ over the interval $[-\infty, 0]$ since $X$ is a non-negative random variable. 
We are therefore entitled to eliminate the absolute value. This means that the integral reduces to 

\begin{eqnarray}
f_T(t) 	& =&  \int_{0}^{\infty} xf_U(xt)f_X(x)dx   \\
		& = & \int_{0}^{\infty} x \frac{1}{\sqrt{2\pi}}\exp \Big \{{-\frac{(xt)^2}{2}}\Big\}\frac {2^{1-\frac k 2}}{\Gamma\left(\frac {k}{2}\right)} k^{\frac k 2}x^{k-1} \exp\Big \{{-\frac {k}{2}x^2}  \Big\}dx \\
		& = & \frac{1}{\sqrt{2\pi}}\frac {2^{1-\frac k 2}}{\Gamma\left(\frac {k}{2}\right)} k^{\frac k 2}\int_{0}^{\infty} x^k \exp \Big \{-\frac 12 (k +t^2) x^2\Big\} dx \label{proof1:student3}
\end{eqnarray}

To conclude, we make the following change of variable $x=\sqrt{\frac w k}$ that leads to
\begin{equation}
f_T(t) =   \frac{1}{\Gamma(\frac{k}{2}) 2^{\frac{k+1}{2}}\sqrt{\pi k}} \int_0^\infty e^{-w (\frac{t^2}{2 k}+\frac{1}{2})} w^{\frac{ k-1}{2} }  dw 
\end{equation}
\qed

\subsection{Proving step 2}
The first step is quite relevant as it proves that the integral to compute takes various form depending on the change of variable done.
\subsubsection{Using Gamma function}
Using the change of variable $w = \frac{ 2 k u}{t^2 + k }$ and knowing that $\Gamma(n) =\int_0^\infty e^{-u}u^{n-1}\ du$, we can easily conclude as follows:
\begin{eqnarray} 
f_T(t)  &=& \frac{1}{\Gamma(\frac{k}{2}) 2^{\frac{k+1}{2}}\sqrt{\pi k}} \int_0^\infty e^{-w (\frac{t^2}{2 k}+\frac{1}{2})} w^{\frac{ k-1}{2} }  dw  \\
&= &\frac{1}{\Gamma(\frac{k}{2})2^{\frac{k+1}{2}}\sqrt{\pi k}} \bigg(\frac{2 k}{t^2+k} \bigg)^{\frac{k+1}{2}}\int_0^\infty e^{-u}u^{\frac{k+1}{2}-1}\ du  \\
&= & \frac{1}{\Gamma(\frac{k}{2})2^{\frac{k+1}{2}}\sqrt{\pi k}} \bigg(\frac{2 k}{t^2+k} \bigg)^{\frac{k+1}{2}}\Gamma\Big(\frac{k+1}{2}\Big) \\
& =& \frac{\Gamma(\frac{k+1}{2})}{\Gamma(\frac{k}{2})}\frac{1}{\sqrt{\pi k}}\bigg( \frac{k }{t^2+ k}\bigg)^{\frac{k+1}{2}}
\end{eqnarray}

\qed

\subsubsection{Using Gamma distribution properties}
Another way to conclude is to notice the kernel of a gamma distribution pdf given by $x^{\alpha-1}\,e^{x\,\lambda}$ in the integral of \ref{t_student:step1_eq1} 
with parameters $\alpha=(k+1)/2,\,\lambda=(1/2)(1+t^2/k)$. The generic pdf for the gamma distribution is $\large \frac{\lambda^\alpha}{\Gamma(\alpha)}\,x^{\alpha-1}\,e^{x\,\lambda}$ and it sums to one over $\left[ 0, \infty \right]$, hence

\begin{eqnarray} 
f_T(t)  &=& \frac{1}{\Gamma(\frac{k}{2}) 2^{\frac{k+1}{2}}\sqrt{\pi k}} \int_0^\infty e^{-w (\frac{t^2}{2 k}+\frac{1}{2})} w^{\frac{ k-1}{2} }  dw  \\
& = & \frac{1}{\Gamma(\frac{k}{2}) 2^{\frac{k+1}{2}}\sqrt{\pi k}}  \frac{ \Gamma(\frac{k+1}{2}) }{  (\frac{t^2+ k}{2 k})^{\frac{k+1}{2} } } \\
& =& \frac{\Gamma(\frac{k+1}{2})}{\Gamma(\frac{k}{2})}\frac{1}{\sqrt{\pi k}}\bigg( \frac{k }{t^2+ k}\bigg)^{\frac{k+1}{2}}
\end{eqnarray}
\qed

\subsubsection{Using Mellin transform}
The integral of equation \ref{t_student:step1_eq1} can be seen as a Mellin transform for the function $g(x) = e^{-w (\frac{t^2}{2 k}+\frac{1}{2})} $, whose solution is well known and given by 
\begin{eqnarray}
\mathcal{M}_g(\frac{ k+1}{2} )  \equiv \int_0^{\infty} x^{\frac{ k+1}{2} -1 } g(x) dx  = \frac{ \Gamma(\frac{k+1}{2}) }{  (\frac{t^2+ k}{2 k})^{\frac{k+1}{2} } }  \label{proof1:student4} 
\end{eqnarray}
Like previously, this concludes the proof. \qed

\subsubsection{Using Laplace transform}
We can use a result of Laplace transform for the function $ f(u) = u^{\alpha}$ as folllows:
\begin{equation}
\mathcal{L}_{f}(s) = \int_0^\infty e^{-u s }u^{\alpha} du  = \frac{ \Gamma(\alpha + 1 ) }{ s ^{ \alpha + 1 }}
\end{equation}

Hence the integral $\int_0^\infty e^{-u}u^{\frac{k+1}{2}-1}\ du$ is simply the the value of the Laplace transform of the polynomial function taken for $s=1$, whose value is $\Gamma\Big(\frac{k+1}{2}\Big)$. Making the change of variable $w = \frac{ 2 k u}{t^2 + k }$ in equation \ref{t_student:step1_eq1} enables to conclude similarly to the proof for the Gamma function \qed

\subsubsection{Using other transforms}
Indeed, as the Laplace transform is related to other transform, we could also prove the result with Laplace–Stieltjes, Fourier, Z or Borel transform.

\subsection{Sum of independent normals} \label{sum_indep_normal}
We want to prove that if $X_i \sim N(0,1)$ then $\frac 1 {n-1} \sum_{i=1}^n (X_i-\bar X_n)^2 \sim \chi^2_{n-1}$. There are multiple proofs for this results:
\begin{itemize}
\item Recursive derivation
\item Cochran's theorem
\end{itemize}

\subsubsection{Recursive derivation}

\begin{lemma}
Let us remind a simple lemma:
\begin{itemize}
\item If $Z$ is a $N(0, 1)$ random variable, then $Z^2 \sim \chi^2_1$; which states that the square of a standard normal random variable is a chi-squared random variable.
\item If $X_1, \ldots, X_n$ are independent and $X_i \sim \chi^2_{p_i}$ then $X_1 + \ldots + X_n \sim  \chi^2_{p_1+ \ldots + p_n}$, which states that independent chi-squared variables add to a chi-squared variable with its degree of freedom equal to the sum of individual degree of freedom.
\end{itemize}
\end{lemma}
The proof of this simple lemma can be established with variable transformations for the fist part and by moment generating function for the second part.
We can now prove the following proposition

\begin{proposition}
If $X_1, \ldots, X_n$ is a random sample from a $N( \mu, \sigma^2)$ distribution, then
\begin{itemize}
\item $\bar{X}_n$ and $s_n^2$ are independent random variables.
\item $\bar{X}_n$  has a $N( \mu, \sigma^2 / n )$ distribution where $N$ denotes the normal distribution.
\item  $(n-1) s_n^2 / \sigma^2 $ has a chi-squared distribution with $n - 1$ degrees of freedom.
\end{itemize}
\end{proposition}

\begin{proof} 
Without loss of generality, we assume that $\mu = 0$ and $\sigma  = 1$. We first show that $s_n$ can be written only in terms of $\left( X_i - \bar{X}_n \right)_{i=2, \ldots, n}$. 
This comes from:

\begin{eqnarray}
s_n^2 & = & \frac{1}{n-1}\sum_{i=1}^{n}(X_i - \bar{X}_n)^2 = \frac{1}{n-1} \left[ (X_1 - \bar{X}_n)^2  + \sum_{i=2}^{n}(X_i - \bar{X}_n)^2 \right] \\
& = & \frac{1}{n-1} \left[ (\sum_{i=2}^{n}(X_i - \bar{X}_n))^2  + \sum_{i=2}^{n}(X_i - \bar{X}_n)^2 \right] 
\end{eqnarray}
where we have use the fact that $\sum_{i=1}^{n}(X_i - \bar{X}_n)= 0$, hence $X_1 - \bar{X}_n = - \sum_{i=2}^{n}(X_i - \bar{X}_n)$. 

We now show that  $s_n^2$ and $\bar{X}_n$ are independent as follows:
The joint pdf of the sample  $X_1, \ldots, X_n$  is given by

\begin{equation}
f(x_1, \ldots, x_n) = \frac 1 {(2 \pi )^{n/2}} e^{- \frac 1 2  \sum_{i=1}^{n} x_i^2}, \quad -\infty  < x_i < \infty.
\end{equation}

We make the    
\begin{eqnarray}
y_1 &= & \bar x \\
y_2 & = & x_2 - \bar x \\
\vdots \\
y_n & = & x_n - \bar x 
\end{eqnarray}

The Jacobian of the transformation is equal to $1/n$. Hence
\begin{eqnarray}
f(y_1, , \ldots, y_n) & =&  \frac n {(2 \pi )^{n/2}}  e^{- \frac 1 2 (y_1 - \sum_{i=2}^{n} y_i)^2 }  e^{- \frac 1 2  \sum_{i=2}^{n} (y_i + y_1)^2}, \quad -\infty  < x_i < \infty \\
& = & [ (\frac n {2 \pi} ) ^{1/2}  e^{- \frac n 2 y_1^2 } ] [  \frac {n^{1/2}}{ (2 \pi) ^{(n-1)/2} } e ^{ - \frac 1 2  [ \sum_{i=2}^{n} y_i^ 2 +  (\sum_{i=2}^{n}  y_i)^2 ] } ] 
\end{eqnarray}

which proves that $Y_1 = \bar X_n$ is independent of $Y_2, \ldots,  Y_n$, or equivalently,  $\bar{X}_n$ is independent of $s_n^2$. To finalize the proof, we need to derive a recursive equation for $s_n^2$ as follows:
We first notice that there is a relationship between $\bar x_n$ and $\bar x_{n+1}$ as follows:

\begin{eqnarray}
\bar x_{n+1} =  \frac{\sum_{i=1}^{n+1} x_i }{n + 1} = \frac{x_{n+1} + n \bar{x}_n}{n + 1} = \bar x_n + \frac{1}{n + 1} (x_{n+1}- \bar x_n),
\end{eqnarray}

We have therefore:

\begin{eqnarray}
n s^2_{n+1} &= & \sum_{i=1}^{n+1} (x_i - \bar x_{n+1})^2 = \sum_{i=1}^{n+1} [  (x_i - \bar x_{n}) - \frac{1}{n+1}( x_{n+1} - \bar x_n) ] ^2 \\
&=& \sum_{i=1}^{n+1} [  (x_i - \bar x_{n})^2 - 2 (x_i - \bar x_n) ( \frac{ x_{n+1}-\bar x_n }{n+1} ) + \frac{1}{(n+1)^2} (x_{n+1} - \bar x_n)^2 ] \\
&=& \sum_{i=1}^{n+1}  (x_i - \bar x_{n})^2  +   (x_{n+1} - \bar x_{n})^2 - 2 \frac{(x_{n+1} - \bar x_n)^2}{n+1} + \frac{(n + 1)}{(n + 1)^2} (x_{n+1} - \bar x_n)^2 \\
&=& (n - 1) s_{n}^2 + \frac{n}{n + 1} (x_{n+1} - \bar x_n)^2
\end{eqnarray}

We can now get the result by induction. The result is true for $n=2$ since $s_2^2 = \frac{ (x_2 -x_1)^2}{2}$ with $ \frac{ x_2 -x_1}{\sqrt 2} \sim N(0,1)$, hence $s_2^2 \sim \chi^2_1$. Suppose it is true for $n$, that is $(n - 1) s_{n}^2 \sim \chi^2_{n-1}$, then since $n s^2_{n+1} = (n - 1) s_{n}^2 + \frac{n}{n + 1} (x_{n+1} - \bar x_n)^2$, $s^2_{n+1}$ is the sum of a $\chi^2_{n-1}$ and $ \frac{n}{n + 1} (x_{n+1} - \bar x_n)^2$ which is independent of $s_{n}$ and distributed as $\chi^2_1$ since $x_{n+1} - \bar x_n\sim N(0, \frac{n+1}{n}$. Using our lemma, this means that $n s^2_{n+1} \sim \chi_{n}^2$. This concludes the proof. 
\end{proof}

\subsubsection{Cochran's theorem}
\begin{proof}
We define the sub vectorial space $F$ spanned by the vector $\mathbbm{1}_n = (1, \ldots, 1)^2$ which is one for each coordinate. 
Its projection matrix is given by $P_F = \mathbbm{1}_n (\mathbbm{1}_n^T \mathbbm{1}_n)^{-1} \mathbbm{1}_n^T = \frac 1 n \mathbbm{1}_n \mathbbm{1}_n^T$. 
The orthogonal sub vectorial space of $\mathbb{R}_n$, denoted by $F^{\bot}$, 
has its projection matrix given by $P_{F^{\bot}} = Id_n - P_F$. 
The projection of the $(x_i)_{i=1, \dots, n}$ over $F$ (respectively $F^{\bot}$) is given by $(\hat x_n, \ldots, \hat x_n)^T$ (respectively) $(x_1 - \hat x_n, \ldots, x_n - \hat x_n)^T$.
 The Cochran's theorem states that these two vectors are independent and that $|| P_{F^{\bot}} X ||^2 = (n-1) s_n ^2 \sim \chi^2(n-1)$
\end{proof}

\section{Various proof around Normal}

\subsection{Linear combination of Correlated Normal} \label{corr_normal}
For any $d$-dimensional multivariate normal distribution $X\sim N_{d}(\mu,\Sigma)$ where $N_d$ stands for the multi dimensional normal distribution,  
 $\mu=(\mu_1,\dots,\mu_d)^T$ and $\Sigma_{jk}=cov(X_j,X_k)\;\;j,k=1,\dots,d$, the characteristic function is given by:
\begin{eqnarray}
\varphi_{X}({\bf{t}}) & =  & E\left[\exp(i{\bf{t}}^TX)\right]=\exp\left(i{\bf{t}}^T\mu-\frac{1}{2}{\bf{t}}^T\Sigma{\bf{t}}\right) \\
& =& \exp\left(i\sum_{j=1}^{d}t_j\mu_j-\frac{1}{2}\sum_{j=1}^{d}\sum_{k=1}^{d}t_jt_k\Sigma_{jk}\right)
\end{eqnarray}

For a new random variable $Z={\bf{a}}^TX=\sum_{j=1}^{d}a_jX_j$, the characteristic function for $Z$ writes:
\begin{eqnarray}
\varphi_{Z}(t)& = & E\left[\exp(itZ)\right]=E\left[\exp(it{\bf{a}}^TX)\right]=\varphi_{X}(t{\bf{a}}) \\
&= & \exp\left(it\sum_{j=1}^{d}a_j\mu_j-\frac{1}{2}t^2\sum_{j=1}^{d}\sum_{k=1}^{d}a_ja_k\Sigma_{jk}\right)
\end{eqnarray}

This proves that  $Z$ is normally distributed with mean given by
$\mu_Z=\sum_{j=1}^{d}a_j\mu_j$ and variance given by $\sigma^2_Z=\sum_{j=1}^{d}\sum_{k=1}^{d}a_ja_k\Sigma_{jk}$.  We can simplify the expression for the variance since $\Sigma_{jk}=\Sigma_{kj}$ get:

\begin{equation}
\sigma^2_Z=\sum_{j=1}^{d}a_j^2\Sigma_{jj}+2\sum_{j=2}^{d}\sum_{k=1}^{j-1}a_ja_k\Sigma_{jk}
\end{equation}
\qed

\subsection{Variance of the sample mean in AR(1) process} \label{proof_var_num}
The computation is given as follows
\begin{eqnarray}
\text{Var}(\sqrt n (\bar{X}_n-\mu)) &= & \frac{1}{n}\text{Var}( \sum_{i=1}^{n} (X_i-\mu) ) = \frac{1}{n} \mathbb{E} \left[ ( \sum_{i=1}^{n} (X_i - \mu)) ^2 \right] \\
&=&   \frac{1}{n} \mathbb{E} \left[ \sum_{i=1}^{n} (X_i- \mu )  ^2 + 2 \sum_{i=1..n, j=1...i-1} (X_i - \mu)  (X_j -  \mu) \right] 
\end{eqnarray}

We have
\begin{eqnarray}
\mathbb{E} \left[ \sum_{i=1}^{n} (X_i- \mu )^2 \right]   & = & \frac{ n \sigma^2}{(1 - \rho^2)} 	\\
\nonumber \\ 
\text{and} \quad  \quad \mathbb{E} \left[  \sum_{i=1..n, j=1...i-1} (X_i - \mu)  (X_j -  \mu)\right] & = &  \frac{ \sigma^2}{(1 - \rho^2)} \sum_{\substack{i=1..n \\ j=1...i-1}} \rho^{i-j} \\
& = &  \frac{ \sigma^2}{(1 - \rho^2)} \sum_{i=1..n} (n-i) \rho^{i} 
\end{eqnarray}

At this stage, we can use rules about geometric series. We have 
\vspace{-0.1cm}

\begin{align}
\sum_{i=1}^{n} \rho^{i}  	& =\rho \sum_{i=0}^{n-1} \rho^{i}  	& \sum_{i=1}^{n} i \rho^{i} 	& =  \rho  \frac{\partial}{\partial \rho} \sum_{i=0}^{n} \rho^{i} \\
						& =\rho \frac{1-\rho^n}{1-\rho}		&   							& = \rho \frac{  1 - (n+1) \rho^{n} + n \rho^{n+1}}{(1-\rho)^2} 
\end{align}

This leads in particular to
\begin{align}
\sum_{i=1}^{n} (n-i ) \rho^{i} = \rho \frac{ n (1-\rho) - (1-\rho^{n}) }{(1-\rho)^2}
\end{align}

Hence,
\begin{eqnarray}
\text{Var}(\sqrt n (\bar{X}_n-\mu)) &= &\frac{ \sigma^2}{(1 - \rho^2)n} \left[ n + 2 \left( n  \rho \frac{1-\rho^{n}}{1-\rho} - \rho \frac{ 1- (n+1) \rho^{n} + n \rho^{n+1}}{(1-\rho)^2} \right) \right] \\
&=&  \frac{ \sigma^2}{(1 - \rho^2)n} \left[ n + 2 \left(  \frac{ n \rho (1  - \rho) - \rho (1-\rho^{n})}{(1-\rho)^2} \right) \right] \\
&=&  \frac{ \sigma^2}{1 - \rho^2} \left[ \frac{ 1 + \rho}{1-\rho} - \frac{2 \rho (1 -  \rho^{n}) }{n(1-\rho)^2}  \right] \\
&=&  \frac{ \sigma^2}{(1 - \rho)^2} \left[1 - \frac{2 \rho (1 -  \rho^{n}) }{n(1-\rho)(1+\rho)}  \right] 
\end{eqnarray}

\qed

\subsection{Variance of the sample mean in AR(1) process} \label{proof_covar}

The computation of \ref{lemma_covar_eq1} is easy and given by
\begin{eqnarray}
\text{Cov}(\bar{Y}_n, Y_j) & =&  \frac{1}{n} \mathbb{E}\left[ \sum_{i=1}^{n}Y_i Y_j \right]  \\
& = &   \frac{\sigma^2}{ n (1 - \rho^2)} \left[ \sum_{i=1}^{n} \rho^{  \vert i - j  \vert }  \right]  \\
& =& \frac{ \sigma^2}{ n (1 - \rho^2)} \left[ \sum_{i=0}^{n-j} \rho^{i} + \sum_{i=0}^{j-1} \rho^{i} -1  \right]   \\
&= & \frac{ \sigma^2}{ n (1 - \rho^2)} \left[ \frac{ 1-\rho^{n+1-j}} {1-\rho} + \frac{ 1-\rho^{j}} {1-\rho} -1  \right]  \\
&= & \frac{ \sigma^2}{ n (1 - \rho^2)} \left[ \frac{ 1 + \rho -\rho^{n+1-j}-\rho^{j}} {1-\rho}   \right]  
\end{eqnarray}

The second equation \ref{lemma_covar_eq2} is trivial as $$\bar{Y}_n = \frac{\sum_{i=1}^n Y_i}{n}.$$Hence
\begin{eqnarray}
\sum_{j=1}^n \text{Cov}(\bar{Y}_n, Y_j) =  n \text{Var}(\bar{Y}_n)  = \frac{ \sigma^2}{ (1 - \rho)^2} \left[ 1 - \frac{2 \rho (1 -  \rho^{n}) }{n(1-\rho)(1+\rho)}  \right] 
\end{eqnarray}

For the last equation, we can compute and get the result as follows

\begin{eqnarray}
\sum_{j=1}^n \left( \text{Cov}(\bar{Y}_n, Y_j) \right)^2 &=&  \sum_{j=1}^n  \frac{ \sigma^4}{n^2 (1 - \rho^2)^2} \left[ \frac{ 1 + \rho -\rho^{n+1-j}-\rho^{j}} {1-\rho}   \right]  ^2  \\
&=&   \frac{ \sigma^4}{n^2 (1 - \rho^2)^2 (1-\rho)^2} \sum_{j=1}^n  \left[  1 + \rho -\rho^{n+1-j}-\rho^{j}  \right]^2
\end{eqnarray}

Expanding the square leads to 
\begin{equation}
\sum_{j=1}^n  \left[  1 + \rho -\rho^{n+1-j}-\rho^{j}  \right]^2   =    \sum_{j=1}^n   ( 1 + \rho )^2 + (\rho^2)^{n+1-j} + (\rho^{2})^{j}- 2 (1+\rho) \rho^{n+1-j} 
- 2 (1+\rho) \rho^j   + 2 \rho^{n+1}   \\
\end{equation}

Denoting by $S$ the summation, computing the different terms and summing them up leads to

\begin{equation}
S  =  n \left( (1+\rho)^2 + 2 \rho^{n+1} \right) + 2 \rho^2 \frac{1-\rho^{2n}}{1-\rho^2} - 4 (1+\rho) \rho \frac{1-\rho^{n}}{1-\rho}
\end{equation}

since
\begin{equation}
 \sum_{j=1}^n \rho^{j}+\rho^{n+1-j} = 2 \rho \frac{1 - \rho^2}{1-\rho}
\end{equation}

Regrouping all the terms leads to

\begin{eqnarray*}
\sum_{j=1}^n \left( \text{Cov}(\bar{Y}_n, Y_j) \right)^2 & =&  \frac{ \sigma^4}{  (1 - \rho^2)^2}  \left[ \frac{ (1+\rho)^2 + 2 \rho^{n+1} } { (1-\rho)^2 } \frac{1}{n} 
-  \frac{4 (1+\rho)^2 \rho (1-\rho^n) - 2 \rho^2 (1-\rho^{2n})}  {(1-\rho)^2 (1-\rho^2)} \frac{1}{n^2} \right]
\end{eqnarray*}
\qed

\subsection{Expectation of denominator}\label{proof_denom_expectation}
Lemma \ref{lemma_var_num}) states that

\begin{eqnarray}
\mathbb{E} \left[ n \bar{Y}_n^2 \right]  = \frac{ \sigma^2}{1 - \rho^2}   \left( \frac{ 1 + \rho}{1-\rho} - \frac{2 \rho ( 1- \rho^{n}) }{n(1-\rho)^2}  \right)
\end{eqnarray}

We can compute as follows:

\begin{eqnarray}
\mathbb{E} {s_n^2} &= & \frac{1}{n-1} \mathbb{E} \left[ \sum_{i=1}^{n} Y_i^2 - n \bar{Y}_n^2 \right] \\
&=&    \frac{ \sigma^2}{(n-1)(1 - \rho^2)}   \left[ n -  \left( \frac{ 1 + \rho}{1-\rho} - \frac{2 \rho ( 1- \rho^{n}) }{n(1-\rho)^2}  \right) \right ] \\ 
&=&    \frac{ \sigma^2}{(n-1)(1 - \rho^2)}   \left[ (n - 1) -  \left( \frac{ 2 \rho }{1-\rho} - \frac{2 \rho (1-  \rho^{n}) }{n(1-\rho)^2}  \right) \right ] \\ 
&=&    \frac{ \sigma^2}{1 - \rho^2}   \left( 1 - \frac{ 2 \rho }{(1-\rho) (n-1)} ( 1 -\frac{ 1-\rho^{n} }{n (1-\rho)}  \right)  \\
&=&    \frac{ \sigma^2}{1 - \rho^2}   \left( 1 - \frac{ 2 \rho }{(1-\rho) (n-1)} + \frac{ 2 \rho( 1-\rho^{n}) }{n (n-1) (1-\rho)^2}  \right)  
\end{eqnarray}
\qed

\subsection{Second moment of denominator}\label{proof_second_moment}
We have 
\begin{eqnarray}
\mathbb{E} {s_n^4}& = & \frac{1}{(n-1)^2} \mathbb{E} \left[ (\sum_{i=1}^{n} Y_i^2 - n \bar{Y}_n^2 )^2\right] \\
 &= & \frac{1}{(n-1)^2} \mathbb{E} \left[ (\sum_{i=1}^{n} Y_i^2)^2 + n^2 \bar{Y}_n^4 -2 n   \bar{Y}_n^2 (\sum_{i=1}^{n} Y_i^2 ) \right] \\
 &= & \frac{1}{(n-1)^2} \mathbb{E} \left[ \sum_{i=1}^{n} Y_i^4 + 2 \sum_{i=1,k=i+1}^{n} Y_i ^2 Y_k^2 + n^2 \bar{Y}_n^4 -2 n   \bar{Y}_n^2 (\sum_{i=1}^{n} Y_i^2 ) \right] 
\end{eqnarray}

We can compute the fourth moment successively. Since both $Y_i$ and $\bar{Y}_n$ are two normal, their fourth moment is three times the squared variance. This gives:
\begin{eqnarray}
\mathbb{E} \left[ \sum_{i=1}^{n} Y_i^4  \right] &= & 3 n \times  \frac{ \sigma^4 }{(1 - \rho^2)^2} \\
\mathbb{E} \left[  n^2 \bar{Y}_n^4 \right] &= & 3 \times  \frac{  \sigma^4}{(1 - \rho^2)^2}  \left[ \frac{1+\rho}{1-\rho}-\frac{2 \rho (1-\rho^{n})}{n(1-\rho)^2} \right]^2 
\end{eqnarray}

The cross terms between $Y_i$ $Y_k$ is more involved and is computed as follows:
\begin{eqnarray}
\mathbb{E} \bigg[ \sum_{\substack{ i=1 \\ k = i+1}}^{n} Y_i ^2 Y_k^2  \bigg] &= & \mathbb{E} \bigg[ \sum_{\substack{ i=1 \\ k = i+1}}^{n} \rho^{2(k-i)} Y_i ^4  + (1- \rho^{2(k-i)}) Y_i^2 (Y_i^{\perp})^2  \bigg] \\
& =& \frac{ \sigma^4 }{(1 - \rho^2)^2}  \bigg[ \sum_{\substack{ i=1 \\ k = i+1}}^{n} 2 \rho^{2(k-i)} + 1  \bigg] \\
& =& \frac{ \sigma^4 }{(1 - \rho^2)^2}  \bigg[ 2 \sum_{i=1}^{n} (n-i) \rho^{2i} + \frac{ n (n-1)}{2}  \bigg] \\
& =& \frac{ \sigma^4 }{(1 - \rho^2)^2}  \bigg[ 2 \rho^2 \frac{ n (1-\rho^2) - (1-\rho^{2n}) }{(1-\rho^2)^2} +  \frac{ n (n-1)}{2}  \bigg] 
\end{eqnarray}

For the cross term between $\bar{Y}_n^2$ and $\sum_{i=1}^{n} Y_i^2 $, we wan use the fact that $\bar{Y}_n$ and $Y_i$ are two correlated Gaussians. 
Remember that for two Gaussians, $\mathbb{E}[U^2 V^2] = \mathbb{E}[U^2] \mathbb{E}[V^2] + 2 (\text{Cov}(U,V))^2$. We apply this trick to get:

\begin{eqnarray}
\mathbb{E} \left[ \bar{Y}_n^2 (\sum_{i=1}^{n} Y_i^2 ) \right] &= & \sum_{i=1}^{n} \mathbb{E} \left[ \bar{Y}_n^2 Y_i^2 \right]  \\
&= & \sum_{i=1}^{n} \mathbb{E} \left[ \bar{Y}_n^2 \right] \mathbb{E} \left[ Y_i^2 \right]  + 2 (\text{Cov}(   \bar{Y}_n, Y_i ))^2  \\
&= & \mathbb{E} \left[ \bar{Y}_n^2 \right]  \sum_{i=1}^{n} \mathbb{E} \left[ Y_i^2 \right]  + 2  \sum_{i=1}^{n}  (\text{Cov}( \bar{Y}_n, Y_i ))^2  
\end{eqnarray}

The first term is given by
\begin{eqnarray}
\mathbb{E} \left[ \bar{Y}_n^2 \right]  \sum_{i=1}^{n} \mathbb{E} \left[ Y_i^2 \right]  & = &  \frac{ \sigma^4}{ (1 - \rho^2)^2 } \left[ \frac{ 1 + \rho}{1-\rho} - \frac{2 \rho (1 -  \rho^{n}) }{n(1-\rho)^2}  \right]  
\end{eqnarray}

The second term is given by
\begin{align}
2  \sum_{i=1}^{n}  (\text{Cov}( \bar{Y}_n, Y_i ))^2  & = &   \frac{ 2 \sigma^4}{  (1 - \rho^2)^2}  \left[ \frac{ (1+\rho)^2 + 2 \rho^{n+1} } { (1-\rho)^2 } \frac{1}{n} 
-  \frac{4 (1+\rho)^2 \rho (1-\rho^n) - 2 \rho^2 (1-\rho^{2n})}  {(1-\rho)^2 (1-\rho^2)} \frac{1}{n^2} \right]
\end{align}

Summing up all quantities leads to

\begin{align}
(n-1)^2 \mathbb{E} {s_n^4} & = \frac{ \sigma^4}{  (1 - \rho^2)^2}  \bigg[ 3n 
+ n (n-1) + 4 \rho^2 \frac{ n (1-\rho^2) - (1-\rho^{2n}) }{(1-\rho^2)^2} 
+ 3 \left[ \frac{1+\rho}{1-\rho}-\frac{2 \rho (1-\rho^{n})}{n(1-\rho)^2} \right]^2 \\
& 
-2 n \left[ \frac{ 1 + \rho}{1-\rho} - \frac{2 \rho (1 -  \rho^{n}) }{n(1-\rho)^2} 
+  \frac{ (1+\rho)^2 + 2 \rho^{n+1} } { (1-\rho)^2 } \frac{2}{n} 
-  \frac{  (1+\rho)^2 \rho (1-\rho^n) - 2 \rho^2 (1-\rho^{2n})}  {(1-\rho)^2 (1-\rho^2)} \frac{8}{n^2} \right]
 \bigg] 
\end{align}

Regrouping all the terms leads to

\begin{align}
\mathbb{E}[s_n^4] & = \frac{\sigma^4}{(1-\rho^2)^2} \frac{1}{(n-1)^2} \left[ n^2-1 + \rho \left(n A_1 + A_2 + \frac{1}{n} A_3 + \frac{1}{n^2} A_4 \right) \right]
\end{align}

with 
\begin{align}
A_1 & = \frac{-4}{1-\rho^2} \\
A_2 & = \frac{- 2 \left(3 + 9 \rho + 11 \rho^2 + 3 \rho^3 + 6 \rho^n + 12 \rho^{n+1} + 6 \rho^{n+2}-2 \rho^{2n+2}\right)} {(1-\rho^2)^2} \\
A_3 & = \frac{ 4 (1 - \rho^n) (1 - 3 \rho + 4 \rho^2 - 8 \rho^{n+1})}{(1 -r)^3 (1 + r)} \\
A_4 & = \frac{12  \rho (1-\rho^{n})^2}{(1-\rho)^4} 
\end{align}

\qed

\subsection{Variance of denominator}\label{proof_var_moment}
The result is obtained from meticulously computing the variance knowing that

\begin{align}
\text{Var}[s_n^4] & = \mathbb{E}[s_n^4] - (\mathbb{E}[s_n^2] )^2
\end{align}

The terms for the part $ \mathbb{E}[s_n^4]$ have already been computed in proposition \ref{second_moment_denom}. As for the term coming from the square of the expectation, they write as:

\begin{align}
(\mathbb{E}[s_n^2] )^2 & =  \frac{ \sigma^4}{(1 - \rho^2)^2}   \left( 1 - \frac{ 2 \rho }{(1-\rho) (n-1)} + \frac{ 2 \rho( 1-\rho^{n}) }{n (n-1) (1-\rho)^2}  \right)^2 \\
& = \frac{ \sigma^4}{(n-1)^2 (1 - \rho^2)^2}   \left( (n-1) - \frac{ 2 \rho }{1-\rho} + \frac{ 2 \rho( 1-\rho^{n}) }{n (1-\rho)^2}  \right)^2
\end{align}

Let us write the square as $E1 = \left( (n-1) - \frac{ 2 \rho }{1-\rho} + \frac{ 2 \rho( 1-\rho^{n}) }{n (1-\rho)^2}  \right)^2 $. We can expand the square as follows

\begin{align}
E1 & =  (n-1)^2  + \frac{ 4 (n-1)  \rho( 1-\rho^{n}) }{n (1-\rho) } -  \frac{ 4 \rho   (n-1) }{1-\rho }+ \frac{ 2 \rho}{(1-\rho)^2} \left(\frac{ 1-\rho^{n}}{n(1-\rho)} -1 \right)^2  
\end{align}

Rearranging the terms leads then to the final result.
\qed

\clearpage

\bibliographystyle{jfe}
\bibliography{mybib}

\begin{thebibliography}{8}
\expandafter\ifx\csname natexlab\endcsname\relax\def\natexlab#1{#1}\fi

\bibitem[{Benhamou(2018)}]{Benhamou_2018_SampleVariance}
Benhamou, E., 2018. A few properties of sample variance. arxiv .

\bibitem[{Bondesson(1974)}]{Bondesson_1974}
Bondesson, L., 1974. Characterizations of probability laws through constant
  regression. Z.Wahrscheinlichkeitstheorie verw. Gebiete pp. 93--115.

\bibitem[{Bondesson(1983)}]{Bondesson_1983}
Bondesson, L., 1983. When is the t-statistic t-distributed. Sankhya, Ser. A pp.
  338--345.

\bibitem[{Efron(1969)}]{Efron_1969}
Efron, B., 1969. Student's t-test under symmetry conditions. J. Amer. Statist.
  Assoc. pp. 1278--1302.

\bibitem[{Fang et~al.(2001)Fang, Yang, and Kotz}]{Fang_2001}
Fang, K., Yang, Z., Kotz, S., 2001. Generation of multivariate distributions by
  vertical density representation. Statistics pp. 281--293.

\bibitem[{Kagan et~al.(1973)Kagan, Linnik, and Rao}]{Kagan_1973}
Kagan, A., Linnik, Y., Rao, C., 1973. Characterization problems in mathematical
  statistics.

\bibitem[{Mauldon(1956)}]{Mauldon_1956}
Mauldon, J., 1956. Characterizing properties of statistical distributions.
  Quart. J. Math. pp. 155--160.

\bibitem[{Mikusheva(2015)}]{Mikusheva_2015}
Mikusheva, A., 2015. Second order expansion of the t-statistics in ar(1)
  models. Econometric Theory 31, 426--448.

\end{thebibliography}

\clearpage

\end{document}